\documentclass[12pt]{article}
\usepackage{graphicx}
\usepackage{amsmath}
\usepackage{amsthm}
\usepackage{amsfonts}
\usepackage{amssymb}
\usepackage{enumerate}
\usepackage{dsfont}
\usepackage{epstopdf}
\usepackage{geometry}
\usepackage[colorlinks=true,linkcolor=blue,urlcolor=blue,citecolor=red, hyperfigures=false]{hyperref}
\usepackage{tikz}

\definecolor{PineGreen}{cmyk}{0.9,0,1,0.40}
\def\green{\color{PineGreen}}

\definecolor{ProcessBlue}{cmyk}{1,0,0,0.40}

\newtheorem*{theorem*}{Theorem}

\newtheorem{example}{Example}

\def\green{\color{PineGreen}}

\newcommand{\N}{\mathbb N}

\newcommand{\R}{\mathbb R}

\def \ep{\epsilon}

\def \FR{{\cal F}}

\def \SR{{\cal S}}
\def \TR{{\cal T}}

\newcommand{\cav}{{\rm cav\ }}

\newtheorem{Theorem}{Theorem}[section]

\newtheorem{Proposition}[Theorem]{Proposition}
\newtheorem{Lemma}[Theorem]{Lemma}

\newtheorem{Remark}[Theorem]{Remark}
\newtheorem{Remarks}[Theorem]{Remarks}

\newcounter{figurecounter}
\newcounter{figurecounter1}
\newcounter{figurecounter2}

\begin{document}

\title{Solving Two-State Markov Games with Incomplete Information on One Side%
\thanks{
The authors acknowledge the support of the COST Action 16228, the European Network of Game Theory.
Ashkenazi-Golan acknowledges the support of the Israel Science Foundation, grant \#520/16.
Solan acknowledges the support of the Israel Science Foundation, grant \#217/17.}}

\author{Galit Ashkenazi-Golan%
\thanks{The School of Mathematical Sciences, Tel Aviv
University, Tel Aviv 6997800, Israel. e-mail: galit.ashkenazi@gmail.com.},
Catherine Rainer%
\thanks{Universit\'{e} de Bretagne Occidentale, 6, avenue Victor-le-Gorgeu, B.P. 809, 29285 Brest cedex, France.
e-mail: Catherine.Rainer@univ-brest.fr.},
and Eilon Solan%
\thanks{The School of Mathematical Sciences, Tel Aviv
University, Tel Aviv 6997800, Israel. e-mail: eilons@post.tau.ac.il.}}

\maketitle

\begin{abstract}
We study the optimal use of information in Markov games with incomplete information on one side and two states.
We provide a finite-stage algorithm for calculating the limit value as the gap between stages goes to 0,
and an optimal strategy for the informed player in the limiting game in continuous time.
This limiting strategy induces an $\ep$-optimal strategy for the informed player, provided the gap between stages is small.
Our results demonstrate when the informed player should use his information and how.
\end{abstract}



\section{Introduction}

In most strategic interactions, the players are not fully informed of the game's parameters, like
their opponents' action sets and payoff functions, and sometimes even their own payoff function and the identity of the opponents.
This observation motivates the study of games with incomplete information, which was incepted in the fifties.
Harsanyi \cite{Harsanyi} introduced the model of Bayesian games, which are one-stage games with incomplete information.
Aumann and Maschler \cite{aum,AM} studied repeated games with incomplete information on one side,
provided an elegant characterization to the value, and described optimal strategies for the players.
The value function of the repeated game with incomplete information on one side turns out to be the concavification of
the value function of the one shot game, which is parameterized by the prior belief of the uninformed player over the state.
In particular, one optimal strategy of the informed player involves a single stage of revelation of information,
and thereafter the informed player does not use his information.
This characterization has been extended to continuous-time games by Cardaliaguet \cite{c1}
(see also Cardaliaguet and Rainer \cite{cr2, cr1}, Gr\"un \cite{gruen}, and Oliu-Barton \cite{oliu}),
and to repeated games with incomplete information on both sides
(see, e.g., Aumann, Maschler, and Stearns  \cite{aum3} and Mertens and Zamir \cite{MZ}).
For recent surveys on the topic, see \cite{AH,MSZ}.

In repeated games with incomplete information, the parameters of the game remain fixed throughout the interaction.
Sometimes these parameters change along the play,
in a way that is independent of the players' actions.
For example, changes in global markets affect local consumers and producers,
who in turn have negligible effect on the global market.
A model that captures this feature is that of \emph{Markov games}.

Two-player zero-sum Markov games with incomplete information on one side have been first studied by Renault \cite{Renault},
who proved the existence of the uniform value; see also Neyman \cite{Ney} and H\"orner, Rosenberg, Solan, and Vieille \cite{HRSV}.
Since the state changes over time,
the optimal strategy typically involves a repeated revelation of information.
Recently Cardaliaguet, Rainer, Rosenberg, and Vieille \cite{CRRV} studied discounted two-player zero-sum Markov games with incomplete information on one side,
where the time duration between stages goes to 0,
and characterized the limit value function and the limit optimal strategy of the informed player;
see also Gensbittel \cite{Gensbittel1, Gensbittel}.
In Cardaliaguet \cite{c3} and Gr\"un \cite{gruen3},
numerical schemes are developed for differential (resp. stochastic differential) games with asymmetric information,
by extending time discretization methods for partial differential equations.

This paper is part of a project whose goal is to study the optimal use of information in dynamic situations of incomplete information,
and to provide an easy to use algorithm for calculating the value and optimal strategies.
We study a simple class of games, namely, Markov games with incomplete information on one side and two states, denoted
$s_1$ and $s_2$.
When the current state is $s_1$ (resp.~$s_2$),
the probability that the state changes is $\pi_1$ (resp.~$\pi_2$).
Player~1 knows when the state changes, while Player~2 does not know it.
Yet players observe each other's actions and have perfect recall,
and thus Player~2 may use past actions of Player~1 to deduce the identity of the current state.

We study the limit value as the gap between stages goes to 0.
Consequently, the discount factor as well as the transition probabilities from one state to the other depend on the gap between stages.
We will provide an algorithm for calculating the limit value of this game, when the gap between stages shrinks to 0,
and the probabilities $\pi_1$ and $\pi_2$ shrink to 0 as well.
The algorithm allows us to propose an $\ep$-optimal strategy for the informed player, as long as the gap between stages is sufficiently small.
The optimal strategy of the informed player alternates between two types of behavior: phases in which no information is revealed
and phases in which information is revealed,
and the algorithm allows us to pinpoint when this change of behavior occurs.

To get some intuition to the problem we contrast it with the case of repeated games with incomplete information on one side,
in which the state does not change along the play, as studied by Aumann and Maschler \cite{AM}.
Denote by $p_n$ the belief of the uninformed player on the state at stage $n$.
The process $(p_n)_{n \in \N}$ is a martingale that is controlled by the informed player.
By the Martingale Convergence Theorem the process $(p_n)_{n \in \N}$ converges to a limit $p_\infty$,
which implies that as the game evolves information stops being revealed.
In particular under the optimal strategies of the players the stage payoff will converge to $u(p_\infty)$,
the value of the one-shot game in which the state is chosen according to the probability distribution $p_\infty$
and no player is informed of the chosen state.
It can then be proven that the value of the game is the concavification of the function $u$,
that is, the smallest concave function that is larger than or equal to $u$.
This, in turn, implies that the informed player has an optimal strategy in which information is revealed only at the first stage of the play.

In the model that we study there are two states, hence the belief of the player can be summarized by the probability that he assigns to state
$s_1$.
Since the state changes along the play,
the process $(p_n)_{n \in \N}$ is no longer a martingale;
indeed, in addition to its dependence on the informed player's actions,
the belief has a drift towards the stationary distribution of the associated Markov chain,
denoted\footnote{In fact, the complete representation of the distribution is $(p^*,1-p^*)$, meaning $p^*$ is the probability of state $s_1$.} $p^*$.
We study the discounted game, and are interested in the optimal way of information revelation.
As in the case of repeated games with incomplete information on one side,
there are two ways in which the informed player can use his information at stage $n$:
\begin{itemize}
\item[A.1]
The informed player may elect not to reveal any information at that stage.
The optimal payoff is then $u(p_n)$, and the belief changes because of the drift towards $p^*$.
\item[A.2]
The informed player may elect to split the belief between two other beliefs:
for some $p',p'' \in [0,1]$ and some $q \in [0,1]$,
the informed player plays in such a way that $p_{n+1} = p'$ with probability $q$ and $p_{n+1}=p''$ with probability $1-q$.
\end{itemize}
Since we study the limit value as the gap between stages goes to 0,
it will be more convenient to consider the process in continuous time.
One can expect that in continuous time, the interval $[0,1]$ of beliefs will be divided into subintervals, as depicted in Figure~\arabic{figurecounter}.
When the belief is in some subintervals, the informed player will reveal no information, and due to the transition,
the belief of the uninformed player will \emph{slide} towards the invariant distribution $p^*$.
When the belief is in the other subintervals, the informed player will reveal information.
In the latter case, the stage payoff if no information is revealed is low,
hence the informed player will avoid such beliefs.

This gives rise to two types of information revelations on the side of the informed player:
if the current belief is within a subinterval $I = [p',p'']$ which the informed player wants to avoid,
he will \emph{split} the belief of the uninformed player between the two endpoints of the interval, namely, $p'$ and $p''$;
if the current belief is the upper end of this subinterval and $p^* < p'$,
since the belief drifts towards the invariant distribution $p^*$,
the informed player will be able to reveal information in such a way that $p_{n+1} \in \{p',p''\}$.
This implies that the belief will remain $p''$ until it \emph{jumps} to $p'$ at a random time;
while if the  current belief is the lower end of this subinterval and $p^* < p'$,
then $p'$ is the upper end of a subinterval $I'$ in which the informed player reveals no information,
and since the belief drifts towards the invariant distribution $p^*$ it will not get into the subinterval $[p',p'']$.
If the interval lies below the invariant distribution $p^*$,
the behavior of the informed player is mirrored.

\bigskip

\centerline{\includegraphics{figure5.eps}}
\centerline{Figure~\arabic{figurecounter}: Possible information
revelation.}

\bigskip

The above description of the general structure of the optimal information revelation strategy is conjectural.
To prove that this description is correct,
and to provide an algorithm that calculates the value function,
we will write down the equations that a value function that is derived from this description must satisfy,
and use the characterization of the value function as given by Gensbittel \cite{Gensbittel} to show that this intuition is correct.
\addtocounter{figurecounter}{1}

We illustrate the optimal revelation strategies by
 three examples from the seminal work of Aumann and Maschler \cite{AM}, adapted to our model.
For expositional ease, in these examples state $s_2$ is absorbing: once the play reaches it, it remains there forever.
In particular, the invariant distribution is $p^* = 0$.

\begin{example}[Nonrevealing optimal strategy]
\label{example2}

Consider the Markov game with the payoff matrices that appear in Figure~\arabic{figurecounter}.
\setcounter{figurecounter1}{\value{figurecounter}}
\addtocounter{figurecounter1}{1}
For every $p\in[0,1]$ the value of the one-shot game is $u(p)=p(1-p)$ (see \cite{AM}, Section I.2 and the dotted line in Figure~\arabic{figurecounter1}).
The value function is concave, and the informed Player~1 has no incentive to reveal information to Player~2.
Consequently,
the limit optimal strategy consists of playing the myopic optimal strategy.
The belief, which starts at the initial belief, slides towards the invariant distribution $p^*=0$,
and the limit value function is given by the discounted integral of the function $u$ (see the dark line in Figure~\arabic{figurecounter1}).

\begin{center}
\begin{tabular}{|l|c|c|}
\hline
State $s_1$&L&R\\
\hline
T&1&0\\
\hline
B&0&0\\
\hline
\end{tabular}
\ \ \ \ \ \ \ \ \
\begin{tabular}{|l|c|c|}
\hline
State $s_2$&L&R\\
\hline
T&0&0\\
\hline
B&0&1\\
\hline
\end{tabular}
\end{center}
\centerline{Figure~\arabic{figurecounter}: The payoff matrices in Example~\ref{example2}.}
\addtocounter{figurecounter}{1}
\bigskip

\centerline{\includegraphics{figure1.eps}}
\centerline{Figure~\arabic{figurecounter}: The value function in
Example~\ref{example2}.} \addtocounter{figurecounter}{1}

\end{example}

\begin{example}{Revealing optimal strategy}
\label{example1}

Consider the game with the payoff matrices that appear in Figure~\arabic{figurecounter}.
\setcounter{figurecounter1}{\value{figurecounter}}
\addtocounter{figurecounter1}{1}
For every $p\in[0,1]$ the value of the one-shot game is $u(p)=-p(1-p)$ (see \cite{AM}, Section I.3, and the dotted line in Figure~\arabic{figurecounter1}).
The value function is convex, and the informed Player~1 has incentive to reveal his information to Player~2.
Consequently, in the optimal strategy Player~1 reveals his information at
every stage,
and the limit value function is identically 0 (see the dark line in Figure~\arabic{figurecounter1}).

\begin{center}
\begin{tabular}{|l|c|c|}
\hline
State $s_1$&L&R\\
\hline
T&-1&0\\
\hline
B&0&0\\
\hline
\end{tabular}
\ \ \ \ \ \ \ \ \
\begin{tabular}{|l|c|c|}
\hline
State $s_2$&L&R\\
\hline
T&0&0\\
\hline
B&0&-1\\
\hline
\end{tabular}
\end{center}
\centerline{Figure~\arabic{figurecounter}: The payoff matrices in Example~\ref{example1}.}
\addtocounter{figurecounter}{1}
\bigskip

 \centerline{\includegraphics{figure2.eps}}
\centerline{Figure~\arabic{figurecounter}: The value function in Example~\ref{example1}.}
\addtocounter{figurecounter}{1}
\end{example}

We now exhibit a nontrivial and challenging case, where the   function $u$ is neither convex nor concave.

\begin{example}{Partial revelation of information}
\label{example3}

Consider the Markov game with the payoff matrices that appear in Figure~\arabic{figurecounter}.
\setcounter{figurecounter1}{\value{figurecounter}}
\addtocounter{figurecounter1}{1}
For every $p\in[0,1]$ the value of the one-shot game is given by (see \cite{AM}, Section I.4 and the dotted line in Figure~\arabic{figurecounter1})
\begin{equation}
\label{uex3} u(p)=\left\{
\begin{array}{ll}
\frac{9p^2-9p+2}{6p-3}, & \mbox{ if } 0\leq p\leq \frac{1}{3},\\
0&\mbox{ if } \frac{1}{3}< p\leq \frac{2}{3},\\
\frac{9p^2-9p+2}{6p-3} & \mbox{ if } \frac{2}{3}< p\leq 1.
\end{array}\right.
\end{equation}
For $p \leq \tfrac{1}{3}$ the function $u$ is convex, hence it is optimal for Player~1 to reveal some of his information.
He should therefore pick some $p_0 \geq \tfrac{1}{3}$ and split the belief of Player~2 between $p=0$ and $p=p_0$.
Do we have $p_0=\tfrac{1}{3}$ or $p_0 > \tfrac{1}{3}$?

Consider next the case that the initial belief is $p=1$.
If Player~1 reveals no information, at every stage $k$ in which the belief is $p_k$ he obtains the payoff $u(p_k)$,
and the belief drifts towards $0$.
Consequently, his payoff slides down the graph of $u$.
Since in the interval $\tfrac{1}{3} \leq p \leq 1$ the graph of $u$ lies below the line segment that connects the points
$(\tfrac{1}{3},u(\tfrac{1}{3}))$ and $(1,u(1))$,
it is not optimal for Player~1 to hide his information throughout:
when the belief reaches some point $p_1 \in [\tfrac{1}{3},1]$ he should start revealing information.
What is this point $p_1$?
How much information does Player~1 reveal?
We will answer these questions
and provide an algorithm that
describes the limit strategy
in the general case.

\begin{center}
\begin{tabular}{|l|c|c|}
\hline
State $s_1$&L&R\\
\hline
T&1&0\\
\hline
B&0&2\\
\hline
\end{tabular}
\ \ \  \ \ \ \ \ \
\begin{tabular}{|l|c|c|}
\hline
State $s_2$&L&R\\
\hline
T&-2&0\\
\hline
B&0&-1\\
\hline
\end{tabular}
\end{center}
\centerline{Figure~\arabic{figurecounter}: The payoff matrices in Example~\ref{example3}.}
\addtocounter{figurecounter}{1}
\bigskip

\centerline{\includegraphics{figure4.eps}}
\centerline{Figure~\arabic{figurecounter}: The value function in
Example~\ref{example3}.}
\setcounter{figurecounter2}{\value{figurecounter}}
\addtocounter{figurecounter}{1}
\end{example}


The paper is organized as follows.
The model as well as known results appear in Section~\ref{sec model}.
Section \ref{sec alg} details the algorithm,
Section \ref{sec examples} demonstrates the algorithm on few examples,
and Section \ref{sec proof} proves the correctness of the algorithm.

\section{The model}\label{sec model}

In this paper we study two-player zero-sum Markov games, which were first studied in Renault \cite{Renault}.
A two-player zero-sum \emph{Markov game} $G$ is a vector $(S, A, B, g, \delta, \pi_1,\pi_2,p)$ where
\begin{itemize}
\item   $S = \{s_1,s_2\}$ is the set of states.
\item   $A$ and $B$ are finite action sets for the two players.
\item   $g : S \times A \times B \to \R$ is a payoff function.
\item   $\delta$ is the discount rate.
\item   $\pi_1$ and $\pi_2$ are the rates of transition.
\item   $p$ is the prior probability that the initial state is $s_1$.
\end{itemize}

The game is played as follows.
The initial state $s^1$ is chosen according to the probability distribution $[p(s_1),(1-p)(s_2)]$;
that is, the initial state is $s_1$ with probability $p$, and $s_2$ with probability $1-p$.
At every stage $k \in \N$ the players
choose independently and simultaneously actions
$a^k$ and $b^k$ in their action sets.
If $s^k = s_1$, then the new state $s^{k+1}$ is equal to $s_1$ with probability $1-\pi_1$ and to $s_2$ with probability $\pi_1$.
Similarly, if $s^k = s_2$, then the new state $s^{k+1}$ is equal to $s_2$ with probability $1-\pi_2$ and to $s_1$ with probability $\pi_2$.
Player~1 is the maximizer and Player~2 is the minimizer.

For every finite set $Y$, $\Delta(Y)$ denotes the set of probability distributions over $Y$.
We assume that information is asymmetric: Player~1 knows the current state while Player~2 does not.
In addition, we assume perfect recall.
Consequently,
a strategy of Player~1 is a sequence $\sigma=(\sigma_k)_{k\geq 1}$, where $\sigma_k:(S\times A\times B)^{k-1}\rightarrow\Delta(A)$ for every $k\geq 1$.
A strategy for Player~2 is a sequence $\tau=(\tau_k)_{k\geq 1}$,
where $\tau_k:(A\times B)^{k-1}\rightarrow\Delta(B)$ for every $k\geq 1$.
The sets of strategies of Player 1 and Player 2 are denoted by $\SR$ and $\TR$, respectively.
Every pair of strategies $(\sigma,\tau) \in \SR \times \TR$, together with the prior belief $p$,
induces a probability distribution on the space
$(S\times A\times B)^\N$ of \emph{plays}, and the payoff is given by
\[ g(p,\sigma,\tau):=E_{p,\sigma,\tau}\left[\sum_{k\geq 1}\delta(1-\delta)^{k-1}g(s^k,a^k,b^k)\right].\]

The value of the game $G$ is given by
\begin{equation}
\label{equ:v}
v := \max_{\sigma \in \SR} \min_{\tau \in \TR} g(p,\sigma,\tau)= \min_{\tau \in \TR} \max_{\sigma \in \SR} g(p,\sigma,\tau).
\end{equation}
The value exists because the payoff is discounted and the strategy spaces of the players are compact in the product topology.
A strategy $\sigma$ (resp.~$\tau$) of Player~1 (resp.~Player~2) that achieves the maximum (resp.~minimum)
in the second (resp.~third) term in Eq.~\eqref{equ:v} is called \emph{optimal}.

We will be interested in the value of the game and in the optimal strategy of Player~1 when the duration between stages is small.
Consequently, we will parameterize the game with a parameter $n > 0$, that will capture the duration between stages.
Thus,
given three positive real numbers $r$, $\lambda_1$, and $\lambda_2$,
we denote by $G^{(n)}(p)$ the Markov game $(S, A, B, g, 1-e^{r/n}, 1-e^{\lambda_1/n},1-e^{\lambda_2/n},p)$.
We denote by $v^{(n)}(p) = v^{(n)}(p,r,\lambda_1,\lambda_2)$ the value of the game $G^{(n)}(p)$.
It follows that
the rates of switching states are roughly $\lambda_1/n$ and $\lambda_2/n$,
and therefore the limit invariant distribution as $n$ goes to infinity is
\[ p^* := \frac{\lambda_2}{\lambda_1+\lambda_2}. \]
Denote also
\[ \mu := \frac{r}{\lambda_1+\lambda_2}. \]

By Cardaliaguet, Rainer, Rosenberg, and Vieille \cite{CRRV} the limit
$v :=\lim_{n \to \infty} v^{(n)}$ exists and the limit as $n$ goes to infinity of the optimal strategy of Player~1 can be characterized as the solution of a certain optimization problem.
We now describe this result.
Extend the domain of the payoff function $g$ to $S \times \Delta(A) \times \Delta(B)$ in a bilinear fashion:
\[ g(s,x,y)=\sum_{a\in A}\sum_{b\in B}g(s,a,b)x(a)y(b), \ \ \ \forall (x,y)\in\Delta(A)\times\Delta(B), s\in S. \]
For $p\in[0,1]$, the value of the one-shot game given that the two states $s_1$ and $s_2$ are observed by none of the players and $s_1$
is the current state with probability $p$ (and $s_2$ with probability $1-p$) is
\[ u(p):=\max_{x\in\Delta(A)}\min_{y\in\Delta(B)}\bigl(pg(s_1,x,y)+(1-p)g(s_2,x,y)\bigr).\]

Let $p \in [0,1]$ be given and let $(\Omega,\FR,P)$ be a sufficiently large probability space.
Let $\SR(p)$ be the set of all c\`adl\`ag, $[0,1]$-valued processes $(p_t)_{t\geq 0}$ defined over $(\Omega,\FR,P)$
that satisfy $E[p_0]=p$
and
$E[p_t|\FR^{p_\cdot}_s]=p_se^{-\lambda_1(t-s)}+ (1-p_s)(1-e^{-\lambda_2(t-s)})$ for every $0\leq s\leq t$, where $\FR^{p_\cdot}_t$
is the $\sigma$-algebra generated by $(p_s)_{s \leq t}$.

\begin{Theorem}[\cite{CRRV}, Theorem 1]
\label{theorem:crrv}
The sequence of functions $p\mapsto v^{(n)}(p)$ converges uniformly to a function $v : [0,1] \to \R$ that satisfies
\begin{equation}
\label{P1}
 v(p)=\max_{ (p_t)_{t \geq 0}\in\SR(p)}E\left[\int_0^\infty re^{-rt}u(p_t)dt\right], \ \ \ \forall p \in [0,1].
 \end{equation}
\end{Theorem}

The processes $(p_t)_{t \geq 0}\in\SR(t)$ in Eq.~\eqref{P1} represent the possible revelation mechanisms
induced by the actions of the informed player.
In particular, the process that realizes the maximum in Eq.~\eqref{P1} represents the optimal revelation process
for the continuous-time game.
The characterization of $v$ provided by Cardaliaguet, Rainer, Rosenberg, and Vieille~\cite{CRRV} is via a differential equation,
as summarized by the next result.

\begin{Theorem}[\cite{CRRV} Theorem 1. P2 ]
\label{theorem:crrv}
The limit value function $v$ is the unique viscosity solution of the equation
\[ \min\{ rv(p)-\langle ^t\! Rp,Dv(p)\rangle -ru(p); -\lambda_{max}v(p,D^2v(p))\}=0, \; \forall p\in\Delta(2),\]
where $R=\begin{pmatrix}
-\lambda_1&\lambda_1\\
\lambda_2&-\lambda_2
\end{pmatrix}$ is the generator of the Markov chain and $\lambda_{max}v(p,D^2v(p))$
is the maximal eigenvalue of the restriction of $D^2v(p)$ to the tangent space at $p$ to $\Delta(2)$.
\end{Theorem}

Since in the sequel we will not need any notion of viscosity, we do not provide their definition, and
refer to \cite{CRRV} for the definition used in the above theorem.
In \cite{CRRV} it is also shown how the optimal solution $(p_t)_{t \geq 0}$ in Eq.~\eqref{P1} can be used to identify $\ep$-optimal strategies for the informed player in the discrete-time game $G^{(n)}(p)$,
provided $n$ is sufficiently large.


In \cite{Gensbittel}, Gensbittel reformulates Theorem \ref{theorem:crrv} in terms of directional derivatives.
Using the fact that in the two-state case the resulting equations are one-dimensional,
we can prove that the limit value function $v$ is differentiable on $[0,1]\setminus\{ p^*\}$.
This leads to the following simple characterisation of $v$ that involves only an ordinary differential equation.

Recall that the \emph{hypograph} of a function
$f : [0,1] \to \R$ is the set of all points that are on or below the graph of the function.
When $f$ is concave, its hypograph is a convex set, and its set of extreme points coincides with the set of points on the graph of $f$ where $f$ is not affine,
plus the corner points $(0,v(0))$ and $(1,v(1))$.

\begin{Theorem}
\label{Char}
The function $v$ is the unique continuous, concave function $v : [0,1]\to\R$ which is differentiable on $[0,1]$ except, possibly,
at $p^*$, and that satisfies the following conditions:
\begin{itemize}
\item[G.1] $v(p^*)\geq u(p^*)$, with an equality if $(p,v(p))$ is an extreme point of the hypograph of $v$.
\item[G.2] For every $p\in[0,1]\setminus\{ p^*\}$ we have $v'(p)(p-p^*) + \mu\left(v(p)-u(p)\right)\geq 0$.
\item[G.3] For every extreme point $(p,v(p))$ of the hypograph of $v$ such that $p \neq p^*$ we have
\begin{equation}
\label{eqv}
v'(p)(p-p^*) + \mu\left(v(p)-u(p)\right)= 0,
\end{equation}
\end{itemize}
where for $p=0$ (resp.~$p=1$), $v'(p)$ stands for the right (resp.~left) derivative of $v$ at $p$.
\end{Theorem}

\begin{proof}
By Theorem 2.12 in Gensbittel~\cite{Gensbittel}, the limit value function is the unique concave, Lipschitz function that satisfies
\[
r(v(p)-u(p))-\overrightarrow{D}V(\pi,^t\!\! R\pi)\geq 0, \ \ \ \forall p\in[0,1], \pi=(p,1-p),\]
and, if $\pi = (p,v(p))$ is an extreme point of the hypograph of $v$,
\begin{equation}
\label{equ:21}
r(v(p)-u(p))-\overrightarrow{D}V(\pi,^t\!\! R\pi)\leq 0,
\end{equation}
where $\overrightarrow{D}V(\pi,\cdot)$ is the directional derivative of $V:\Delta(2)\ni (p,1-p)\mapsto V(p,1-p):=v(p)$, and $R=\begin{pmatrix}
-\lambda_1&\lambda_1\\
\lambda_2&-\lambda_2
\end{pmatrix}$.
It follows that
\begin{equation}
\label{equ:22}
\overrightarrow{D}V(\pi,^t\!\! R\pi)=\left\{\begin{array}{ll}
-v'_+(p)\frac{r}{\mu}(p-p^*)&\mbox{ if }p<p^*,\\
-v'_-(p)\frac{r}{\mu}(p-p^*)&\mbox{ if }p>p^*,\\
\overrightarrow{D}V(\pi,0)=0& \mbox{ for } \pi=(p^*,(1-p^*)).
\end{array}
\right.
\end{equation}
The theorem will follow once we show that $v'_-(p)=v'_+(p)$ for every $p\in(0,1)\setminus\{p^*\}$.\\
Suppose first that $p>p^*$. This implies that for every $q>p$ it holds that $q>p^*$, hence $\frac{r}{\mu}(q-p^*)>0$. Since $v$ is concave, it also implies that $v'_-(q)\leq v'_+(p)\leq v'_-(p)$.
From Eqs.~\eqref{equ:21} and~\eqref{equ:22} it follows that
\begin{eqnarray*}
0\leq \mu(v(q)-u(q))+v'_-(q)(q-p^*)
\leq \mu(v(q)-u(q))+v'_+(p)(q-p^*).
\end{eqnarray*}
From the continuity of $u$ and $v$
we deduce that
\[ 0 \leq \mu(v(p)-u(p))+(p-p^*)v'_+(p)\leq \mu(v(p)-u(p))+(p-p^*)v'_-(p).\]
If $\pi=(p,v(p))$ is an extreme point of the hypograph of $v$, then
\[  \mu(v(p)-u(p))+(p-p^*)v'_-(p)=0,\]
and it follows that $v'_-(p)=v'_+(p)$.
If $p$ is not an extreme point of the hypograph of $v$, then there exist $p_1,p_2\in[0,1]$ and $\alpha\in(0,1)$ such that
$p=\alpha p_1+(1-\alpha )p_2$ and $v(p)=\alpha v(p_1)+(1-\alpha )v(p_2)$.
Since $v$ is concave, it follows that $v$ is affine on the interval $[p_1,p_2]$, and therefore differentiable on its interior.
In particular, $v'_-(p)=v'_+(p)$ in this case as well.

Suppose now that $p<p^*$. In this case for $q<p$ we have $\frac{r}{\mu}(q-p^*)<0$ and $v'_+(p)\leq v'_-(p)\leq v'_+(q)$ for every $q<p$,
and an analogous argument to the one provided above leads to the same result: $v'_-(p)=v'_+(p)$.
\end{proof}

\begin{Remark}
The arguments of the proof cannot be used for $p=p^*$, because $^t\! Rp^*=0$.
In fact, the function $v$ may not be differentiable at $p^*$, see variation b of Example 3 below.
\end{Remark}

\section{An algorithm to calculate the value function and the optimal revelation process}\label{sec alg}

In this section we present a finite stage recursive algorithm for calculating the limit value function
and the limit optimal strategy for the informed player.
We start by explaining the intuition behind the algorithm.

\subsection{Intuition}

 We shall see that the limit value at the invariant distribution $p^*$ can be explicitly calculated.
The algorithm will assume that the limit value was already calculated in a certain closed interval that contains $p^*$,
and will calculate it for a larger interval.
The calculation for beliefs smaller than $p^*$ will be analogous to the calculation for beliefs larger than $p^*$,
hence we will concentrate on the latter.

In this section we provide the equations that the limit value function must satisfy under the three types
of information revelation that were discussed in the introduction.


\subsubsection{No revelation of information}
\label{strategy:1}

We first provide the equation that the limit value function satisfies in an interval in which no information is revealed by the informed player.
Fixing the time step $1/n$ and the initial distribution $p\in[0,1]$,
 let  $v^{(n)}(p)$ be the value of the corresponding game.
Let $p^*<p'<p''\leq 1$ or $0\leq p'<p''<p^*$, and suppose that, given any belief $p\in[p',p'']$ of the uninformed player, the optimal strategy of the informed player is not to reveal his information.
The continuation payoff is given by
$v^{(n)}(pe^{-\lambda_1/n}+(1-p)(1-e^{-\lambda_2/n}))$.
Therefore the value function satisfies the relation
\begin{equation}\label{eq varphi}
v^{(n)}(p)=(1-e^{-r/n})u(p)+e^{-r/n}v^{(n)}(pe^{-\lambda_1/n}+(1-p)(1-e^{-\lambda_2/n})).
\end{equation}
Simple algebraic manipulations yield that
\begin{eqnarray*}
    &&\frac{v^{(n)}(p)-v^{(n)}\left(pe^{-\lambda_1/n}+(1-p)(1-e^{-\lambda_2/n})\right)}{p-\left(pe^{-\lambda_1/n}+(1-p)(1-e^{-\lambda_2/n})\right)}\\
    &&=
    \frac{(1-e^{-r/n})u(p)
    - (1-e^{-r/n})v^{(n)}\left(pe^{-\lambda_1/n}+(1-p)(1-e^{-\lambda_2/n})\right)}{p-\left(pe^{-\lambda_1/n}+(1-p)(1-e^{-\lambda_2/n})\right)}.
\end{eqnarray*}
Taking the limit as $n$ goes to $\infty$ (recall that $\mu=\frac{r}{\lambda_1+\lambda_2}$ and $p^*=\frac{\lambda_2}{\lambda_1+\lambda_2}$) we obtain that
$v=\lim_nv^{(n)}$ is the solution of the following differential equation:
\begin{eqnarray}
\label{equ:noinformation}
    v'(p) &=& \frac{\mu(u(p)-v(p))}{p-p^*}.
\end{eqnarray}

\subsubsection{Jumping from $p'$ to $p''$ }
\label{strategy:2}

Suppose now the informed player wants to avoid beliefs in some open interval $(p',p'')$, with $p^*<p'$,
and that moreover in some interval $(p'',p''')$ the informed player revealed no information.
As explained in the introduction,
when the belief is $p''$, the informed player will reveal the amount of information that ensures that the belief at the next stage is either $p'$ or $p''$.
We will also assume that if $p \in (p',p'')$, then when the belief is $p$ the informed player splits the belief between $p'$ and $p''$.
It follows that for every $n \in \N$
the value function $v^{(n)}$ is  affine on the interval $[p',p'']$,
and therefore so is the limit value function $v$, that is,
\begin{equation}
\label{v'p2a}
v'_-(p'')=\frac{v(p'')-v(p')}{p''-p'}.
\end{equation}
Moreover, since $p''$ is an endpoint of an interval in which no information is revealed, by Eq.~\eqref{equ:noinformation}
\begin{equation}
\label{v'p2}
v'_+(p'')=\frac{\mu(u(p'')-v(p''))}{p''-p^*}.
\end{equation}
Since $v$ is smooth at $p''$, we have $v'_+(p'')=v'_-(p'')$, and therefore by Eqs.~\eqref{v'p2a} and~\eqref{v'p2} we have
\begin{equation*}
v(p'')\left(\frac{1}{p''-p'}+\frac{\mu}{p''-p^*}\right) =\frac{v(p')}{p''-p'} +\frac{\mu u(p'')}{p''-p^*},
\end{equation*}
or, equivalently,
\begin{equation}
\label{equ:21bb}
v(p'') =\frac{v(p')(p''-p^*)+\mu(p''-p')u(p'')}{p''-p^*+\mu(p''-p')}.
\end{equation}
Substituting $v(p'')$ from Eq.~\eqref{equ:21bb} in Eq.~\eqref{v'p2} we obtain that
\begin{equation}
\label{equ:22b}
v'(p'') = \frac{\mu(u(p'')-v(p'))}{ p''-p^*+\mu(p''-p')}.
\end{equation}

\noindent From the affinity of $v$ we deduce that for every $p\in[p',p'']$ we have

\begin{equation}
\label{eq:22c}
v(p)=v(p')+(p-p')\frac{\mu(u(p'')-v(p'))}{ p''-p^*+\mu(p''-p')}.
\end{equation}

\subsubsection{Splitting the belief}

A third possible strategy for the informed player is to split the belief from $p$ to $p'$ and $p''$, where $p' < p < p''$.
In this case, whenever the belief is in the open interval $(p',p'')$ it will be optimal
for the informed player to reveal information in such a way that the belief is either $p'$ or $p''$.
In continuous time this implies that the belief will never be in the open interval $(p',p'')$.

If both $p'$ and $p''$ are smaller than $p^*$, or both are larger than $p^*$,
then, this kind of information revelation will possibly occur only once, at the first stage of the game.

If $p'<p^*<p''$ then this case reduces to the one described in Section~\ref{strategy:2}:
when the belief is $p''$ (resp.~$p'$), it remains at $p''$(resp.~$p'$) until it jumps at a random time to $p'$ (resp.~$p''$).
Applying Eq.~\eqref{equ:21bb} to the jumps from $p'$ to $p''$ and from $p''$ to $p'$,
we obtain two affine equations in $v(p')$ and $v(p'')$.
If for every $p \in (p',p'')$ the informed player splits the belief to $p'$ and $p''$,
we obtain a strategy for the informed player that guarantees a payoff of
\begin{equation}
\label{eq: middle split2}
v(p)=u(p')\frac{(\mu+1)p''-p^*}{(p''-p')(\mu+1)}+u(p'')\frac{p^*-p'(\mu+1)}{(p''-p')(\mu+1)} + p\mu\cdot\frac{u(p'')-u(p')}{(p''-p')(\mu+1)}, \ \ \ p \in [p',p''].
\end{equation}

\begin{Remark}
\label{remark stationary}
Substituting $p=p^*$ in Eq.~\eqref{eq: middle split2}
we obtain:

    \begin{equation}
    v(p^*)= \frac{p''-p^*}{p''-p^*}u(p')+\frac{p^*-p'}{p''-p'}u(p'').
    \end{equation}

\end{Remark}
\subsubsection{Conclusion}

The intuition we presented describes the conjectured behavior of the belief of Player~2
under the optimal strategy of Player~1:
in the first stage the belief may split,
and thereafter the behavior alternates between sliding continuously towards the invariant distribution $p^*$
and jumping at a random time to a belief closer to $p^*$.

To find the points where the behavior of the belief changes,
we will begin from ``the end'', that is, from $p=p^*$, and work our way towards $p=1$
(and then towards $p=0$).
Supposing that the limit value function was already calculated for every belief $p$ in some interval $[p^*,p_0]$,
we compare the incremental value of the two strategies described in Sections~\ref{strategy:1} and~\ref{strategy:2},
find the maximal interval $[p_0,p_1]$ for which
the better
strategy yields a higher increment,
and accordingly extend the definition of the limit value function to the interval $[p^*,p_1]$.
We then
conduct the analogous procedure for $p$'s smaller than $p^*$.

\subsection{The algorithm to compute the limit value function}
\label{section:algorithm}

In this section we present the algorithm that calculates the limit value function.
We will start with some notations.
Given a continuous real valued function $f$ defined on some interval $I\subset[0,1)$,
we define the function $a(\cdot,f) : I \to \R\cup\{+\infty\}$ by

\begin{equation}
\label{equ:91}
a(p,f)
:=\sup_{p' \in (p,1]}\frac{\mu(u(p')-f(p))}{ p'-p^*+\mu(p'-p)}.
\end{equation}

Analogously, if $f$ is defined on some interval $I\subset(0,1]$, we define
 the function $\widetilde{a}(\cdot,f) : I \to \R\cup\{-\infty\}$ by
\[ \widetilde{a}(p,f)
:=\inf_{p'\in [0,p)}\frac{\mu(u(p')-f(p))}{ p'-p^*+\mu(p'-p)}.\]

Note that for $p\neq p^*$ we have
\[ a(p,f)
=\max_{p' \in [p,1]}\frac{\mu(u(p')-f(p))}{ p'-p^*+\mu(p'-p)} \;\mbox{ and }\;\widetilde{a}(p,f)
:=\min_{p'\in [0,p]}\frac{\mu(u(p')-f(p))}{ p'-p^*+\mu(p'-p)}. \]
The function $a(\cdot,f)$  (resp. $\widetilde{a}(\cdot,f)$) is continuous on $I\setminus \{p^*\}$ ,
and if $f(p^*)=u(p^*)$ and $I=[p^*,\hat p]$ (resp. $I=[\hat p, p^*]$) for some $\hat p$, then $f$ is continuous on $I$.

Define also
\begin{equation}
\label{equ:91b}
\left\{\begin{array}{ll}
\rho(p,f):=
\sup\left\{p'\in(p,1] \colon a(p,f)=\frac{\mu(u(p')-f(p))}{ p'-p^*+\mu(p'-p)}\right\}&\mbox{ if }p>p^*,\\
\widetilde\rho(p,f) :=
\inf\left\{p'\in[0,p) \colon \widetilde a(p,f)=\frac{\mu(u(p')-f(p))}{ p'-p^*+\mu(p'-p)}\right\} &\mbox{ if }p<p^*,
\end{array}\right.
\end{equation}
with $\inf\emptyset=1$ and $\sup\emptyset=0$.

To see the motivation  for these definitions, recall the discussion in Section~\ref{strategy:2}.
When Player~1 wants to make the belief of Player~2 jump from some $p' > p$ to $p$,
the value function on the interval $[p,p']$ is affine and given by Eq.~\eqref{equ:21bb}.
In particular, the slope of the value to the left of $p'$ is given by $\frac{\mu(u(p')-v(p))}{ p'-p^*+\mu(p'-p)}$, see Eq.~\eqref{equ:22b}.
To maximize the payoff in a small neighborhood to the right of $p$,
Player~1 will jump to $p$ from some $p'$ that attains the maximum in Eq.~\eqref{equ:91}.
The quantity $a(p,v)$ is defined to be the slope at such optimal belief,
and $\rho(p,v)$ is the largest optimal belief.
The quantities $\widetilde a(p,v)$ and $\widetilde \rho(p,v)$ have analogous interpretations when $p < p^*$.

\medskip

We now present the algorithm, which defines in steps a function $w : [0,1] \to \R$
that is later shown to be the limit value function.
The initial step of the algorithm identifies a closed interval $[\widetilde{p}_0,p_0]$
that includes the stationary distribution $p^*$, on which the calculation of the limit value function is simple.

The algorithm then defines
iteratively an increasing sequence $(p_k)_{k \geq 0}$ of points in the interval $[p_0,1]$;
at the $k$'th
iteration of the algorithm we define the point ${p}_k$ and extend the definition of $w$ to include $(p_k,p_{k+1}]$.
This part of the algorithm terminates when $p_k=1$.
Finally, the algorithm defines
iteratively a decreasing sequence $(\widetilde{p}_k)_{k \geq 0}$ of points in the interval $[0, \widetilde p_0]$
 and extends the definition of $w$ to include $[\widetilde p_{k+1},\widetilde p_k)$.
This part of the algorithm terminates when $\widetilde{p}_k={\green 0}$.

\bigskip

\noindent
{\bf Initialization}:

Let $p_0=\inf\{ p>p^*, (\cav u)(p)=u(p)\}$ and
$\widetilde p_0=\sup\{ p<p^*, (\cav u)(p)=u(p)\}$.
Define a function $w \colon [\widetilde p_0,p_0] \to \R$ as follows:
\begin{itemize}
\item
If $\widetilde p_0=p^*=p_0$, then set $w(p^*) = u(p^*)$.
\item
If $\widetilde p_0<p_0$,
then $w$ is defined
as follows. For every $p\in[\widetilde p_0,p_0]$,
\begin{equation}
\label{split}
w(p):=u(\widetilde{p}_0)\frac{p_0(\mu+1) - p^*}{(p_0-\widetilde{p}_0)(\mu+1)}+u(p_0)\frac{p^*-\widetilde{p}_0(\mu+1)}{(p_0-\widetilde{p}_0)(\mu+1)}
+ p\mu\cdot\frac{u(p_0)-u(\widetilde{p}_0)}{(p_0-\widetilde{p}_0)(\mu+1)}
\end{equation}
(compare this expression with
Eq.~\eqref{eq: middle split2}).
\end{itemize}

\noindent

\noindent{\bf
Increasing part of the algorithm}:

\begin{enumerate}
    \item[I.1.]
    Let $k \geq 0$ and
    suppose that the function $w$ is already defined on the interval $[p_0,p_k]$.
    \item[I.2.]
    If $p_k=1$, the first part of the algorithm terminates;
    go to Step D.1.
    \item[I.3.]
        If $p_k<1$, let $\varphi_k : [p_{k},1] \to \R$ be the solution of the following differential equation:
    \begin{equation}
    \label{varphiA}
    \left\{\begin{array}{ll}
    \varphi_k(p_{k})=w(p_{k}), &  \\
    \varphi_k'(p)=\frac{\mu(u(p)-\varphi_k(p))}{ p-p^*}, & p\in (p_{k},1],
    \end{array}
    \right.
    \end{equation}
    and set
    \begin{equation}\label{eq psi}
    \psi_k(p):=w(p_{k})+(p-p_{k})a(p_{k},w), \ \ \ \forall p\in(p_{k},1].
    \end{equation}

    \item[I.4.]
    If $\rho(p_k,w)>p_k$, define
    \begin{equation}
    \label{equ:i4}
     p_{k+1}:=\rho(p_k,w).
     \end{equation}

    Extend the domain of $w$ to include $(p_k,p_{k+1}]$ by
    \begin{equation}
    \label{equ:w:i4}
     w(p):=\psi_k(p), \ \ \ \forall p \in (p_k,p_{k+1}].
     \end{equation}
    \item[I.5.]
    Otherwise, $\rho(p_k,w)=p_k$. Define
    \begin{eqnarray}\label{equ:i5}
    p_{k+1}:=\inf\{ p>p_{k}\colon~   \rho(p,\varphi_k)>p\},
    \end{eqnarray}
    with $\inf\emptyset=1$.
    Extend the domain of $w$  to include $(p_k,p_{k+1}]$ by
    \[ w(p):=\varphi_k(p), \ \ \ \forall p \in (p_k,p_{k+1}]. \]
    \item[I.6.]
    Increase $k$ by 1 and go to Step I.2.
\end{enumerate}

\noindent
{\bf Decreasing part of the algorithm}:

\begin{enumerate}
\item[D.1.]
    Let $k \geq 0$
 and
suppose that the function $w$ is already defined on the interval $[\widetilde p_k,\widetilde p_0]$.
\item[D.2.]
If $\widetilde{p}_k=0$,
the algorithm terminates.
\item[D.3.]
If $\widetilde{p}_k>0$, let $\varphi_k : [0,\widetilde p_{k}] \to \R$ be the solution of the following differential equation:
\begin{equation}
\label{varphiB}
\left\{\begin{array}{ll}
      \varphi_k(\widetilde p_{k})=w(\widetilde p_{k}), &  \\
      \varphi_k'(p)=\frac{\mu(u(p)-\varphi_k(p))}{ p-p^*}, & p\in [0,\widetilde p_{k}).
      \end{array}
\right.
\end{equation}
Define
\begin{equation}\label{eq psib}
\widetilde{\psi}_k(p):=w(\widetilde{p}_{k})+(p-\widetilde{p}_{k})\widetilde{a}(\widetilde{p}_{k},w), \ \ \ \forall p\in[0,\widetilde{p}_k].
\end{equation}

\item[D.4.]

If $\rho(\widetilde{p}_k,w)<\widetilde{p}_k$, define

\[ \widetilde{p}_{k+1}:=\rho(\widetilde{p}_k,w). \]

Extend the domain of $w$ to include $[\widetilde{p}_{k+1},\widetilde p_k)$ by
\[ w(p):=\widetilde{\psi}(p), \ \ \ \forall p \in [\widetilde{p}_{k+1},\widetilde{p}_{k}). \]

\item[D.5.]
Otherwise, $\rho(\widetilde{p}_k,w)=\widetilde{p}_k$.
Define
$\widetilde{p}_{k+1}:=\sup\{ p<\widetilde{p}_{k}\colon~ \rho(p,\varphi)<p\}$, with $\inf\emptyset=0$.
Extend the domain of $w$  to include $[\widetilde{p}_{k+1},\widetilde{p}_{k})$ by
\[ w(p):=\varphi_k(p), \ \ \ \forall p \in [\widetilde{p}_{k+1},\widetilde{p}_{k}). \]

\item[D.6.]
Increase $k$ by 1 and go to Step~D.2.

\end{enumerate}

The idea is that after the initialization, the algorithm decides for each point $p_k$
whether,
for beliefs slightly above $p_k$, it is optimal for Player 1 to reveal information
or to reveal nothing until the belief reaches $p_{k}$.
The decision is based on comparison of derivatives:
the derivative of $\varphi$, the nonrevealing payoff, is compared to $a(p_k,w)$, the highest possible derivative when splitting.
The strategy that gives the highest derivative is the one that is played, for as long as it's derivative is indeed the higher one.
The changes from a revealing strategy to nonrevealing strategy and vice versa
occur at the points $(p_k)_{k \geq 0}$ and $(\widetilde p_k)_{k \geq 0}$,
where the former lower derivative becomes the higher one.
Since the derivative from the right is equal to the derivative from the left in points where the behavior of the informed player changes,
the corresponding payoff function, and consequently the limit value function, turn out to be differentiable.

On intervals $(p_k,p_{k+1}]$ (resp.~$[\widetilde p_{k+1},\widetilde p_k)$)
where the function $w$ is defined by Step~I.4 (resp.~D.4),
$w$ is linear,
while on intervals $(p_k,p_{k+1}]$ (resp.~$[\widetilde p_{k+1},\widetilde p_k)$)
where the function $w$ is defined by Step~I.5 (resp.~D.5),
$w$ is nonlinear.
We therefore call intervals on which $w$ is defined by Steps~I.4 and~D.4 (resp.~I.5 and~D.5)
\emph{linear intervals} (resp.~\emph{nonlinear intervals}).

\begin{Remarks}
    \label{remark alg}
    \begin{enumerate}
    \item
    In the initialization step,
    under the optimal strategy of Player~1,
    the belief jumps at random times from $\widetilde p_0$ to $p_0$ and back.
    When $p^*$ is an extreme point of this interval, say $p^* = \widetilde p_0$,
substituting $p=p^*$ in Eq.~\eqref{split} yields $w(p^*) = u(p^*)$.
Consequently, in this case at the belief $p^*$ there is no revelation of information.

\item
We can already affirm that on the interval $[\widetilde p_0,p_0]$ the function $w$ coincides with the value function $v$.
        Indeed, by Lemma 2 in \cite{CRRV}, for every $p\in[\widetilde p_0,p_0]$, we have
\[ v(p)=\int_0^\infty e^{-rt}(\cav u)(p^*+(p-p^*)e^{-(\lambda_1+\lambda_2)t})dt,\]
        with $
        (\cav u)(p)= u(\widetilde p_0)+\frac{u(p_0)-\widetilde u(p_0)}{p_0-\widetilde p_0}(p-\widetilde p_0)$.
        This integral can be calculated explicitly  and it coincides with the expression  of $w$ in Eq.~\eqref{split}.

\item
        \label{remark int}
In general Eq.~\eqref{varphiA} does not have an explicit solution.
In the special case that
 $p^*=0$ and $\mu=1$, this equation has an explicit solution, given by
        \[ \varphi_k(p)=\frac{p_{k}}{p}\varphi(p_{k})+\frac{1}{p}\int_{p_{k}}^p u(t)dt.\]

\item
\label{remark:4}
Calculating the limit of the term on the right-hand side of Eq.~\eqref{equ:91} as $p'$ converges to $p$,
we deduce that for every function $f$ we have
        $a(p,f)\geq \mu\cdot\frac{u(p)-f(p)}{p-p^*}$, provided $p \neq p^*$. In particular,
        substituting $f=\varphi_k$, the
         solution of Eq.~\eqref{varphiA}, this gives
        \[a(p,\varphi_k)\geq\varphi_k'(p).\]
        On a nonlinear interval $(p_k,p_{k+1}]$ we can be even more precise:
        for every $p$ such that $\rho(p,\varphi_k)=p$, it follows from the definition of $a(p,\varphi_k)$ that
        \begin{equation}
        \label{eq stpA5} a(p,\varphi_k)=\mu\cdot\frac{u(p)-\varphi_k(p)}{p-p^*}=\varphi_k'(p).
        \end{equation}
        In particular, given that $p_{k+1}=\inf\{ p>p_k \colon \rho(p,\varphi_k)>p\}$,
        Eq.~\eqref{eq stpA5} holds for every
        $p\in(p_k,p_{k+1})$ as well as for the right (resp.~left) derivative of $\varphi_k$ for $p=p_k$ (resp. $p=p_{k+1}$).

    \end{enumerate}
\end{Remarks}

We now state the main theorem of the paper.

\begin{Theorem}
\label{wv}
\begin{enumerate}
\item
For every $k \geq 0$ such that $p_k < 1$ we have $p_k<p_{k+1}$.
\item
For every $k \geq 0$ such that $\widetilde p_k >0$ we have $\widetilde p_{k+1} < \widetilde p_k$.

\item
The algorithm terminates after a finite number of iterations;
that is, there is
$k \geq 0$ such that $p_k = 1$ and there is
 $ k \geq 0$  such that $\widetilde p_k=0$.
\item The function $w$ generated by the algorithm is the limit value function of the game, i.e., $w=v$.
\end{enumerate}
\end{Theorem}

\noindent The proof of Theorem \ref{wv} is relegated to Section \ref{sec proof}, after the algorithm is demonstrated on some examples.

\section{Examples}\label{sec examples}

In this section we illustrate the algorithm on the three examples provided in the Introduction.
We will also analyze two variants of the third example;
the first will illustrate the algorithm when there is more than one iteration,
and the second will show that the limit value function may be nondifferentiable at $p^*$.
Recall that in these examples, the state $s_2$ is absorbing, so that $p^* = 0$ and $\mu=r=1$.

\bigskip
\noindent\textbf{Example~\ref{example2}, continued.}
In this example the function $u$ is given by $u(p) = p(1-p)$ for every $p \in [0,1]$.
The function $u$ is concave, and so
$\widetilde{p}_0=p_0=0$, and $w(0)=0$.
We next have to compute the solution of Eq.~\eqref{varphiA} with initial condition $\varphi(0)=0$.  For
$p\in [0,1]$, it is (see Remark~\ref{remark alg}.\ref{remark int})
\[ \varphi(p)=\frac{0}{p}+\frac{1}{p}\int_{0}^p t(1-t)dt = \frac{p}{2}-\frac{p^2}{3}. \]
It follows that
\[ a(p,\varphi)=\sup_{p'\in (p,1]}\frac{p'(1-p')-\frac{p}{2}+\frac{p^2}{3}}{2p'-p}.\]
For every $p\in[0,1]$ the supremum is obtained only at $p'=p$, that is $\rho(p,\varphi)=p$ for every $p \in[0,1]$.
This implies that the condition of Step~I.5 holds, $p_1=1$,
and the first part of the algorithm terminates.
Since $\widetilde p_0=0$, the second part of the algorithm is vacuous.
In conclusion, the limit value function is given by
\[v(p)=\frac{p}{2}-\frac{p^2}{3}, \; \forall p\in[0,1],\]
and the optimal strategy of Player~1 is never to reveal his information.

\bigskip

\noindent\textbf{Example~\ref{example1}, continued.}
Recall that in this example the function $u$ is given by $u(p)=-p(1-p)$ for every $p \in [0,1]$.
Since
$(\cav u)(p^*)=0=\alpha u(0) + (1-\alpha)u(1),~\forall \alpha\in[0,1]$,
we have $\widetilde{p}_0=0$ and $p_0=1$.
From Eq.~\eqref{split} we obtain that $v(p)=0$ for every $p\in[0,1]$.
Consequently, the optimal strategy of Player~1 is to always reveal his information.

\bigskip

\noindent\textbf{Example~\ref{example3}, continued.}
In this example the function $u(p)$ is
given
by Eq.~\eqref{uex3} and is represented by the dotted line in  Figure~\arabic{figurecounter2}.
For this example the algorithm runs as follows.
Since $p^* = 0$ we have $\widetilde{p}_0=0$.
Simple calculations show that $p_0=\frac{1}{3}$, and from Eq.~\eqref{split} we have $w(p)=-\frac{2}{3}+p$ for every $p\in[0,\frac{1}{3}]$.
On $[\frac{1}{3},\frac{2}{3}]$ the solution of Eq.~\eqref{varphiA} is
$\varphi_0(p)=\frac{1}{3p}v(\frac{1}{3})+\frac{1}{p}\int_{\frac{1}{3}}^p
0\; dx=-\frac{1}{9p}$.
It follows that, for $p\in[\frac{1}{3},\frac{2}{3}]$,

\[\begin{array}{rl}
a(p,\varphi_0)=&\sup_{p'\in (p,1]}\frac{u(p')-\varphi_0(p)}{2p'-p}\\
=&\max\left\{ \sup_{p'\in(p,\frac{2}{3})}\frac{-\varphi_0(p)}{2p'-p};
\sup_{p'\in[\frac{2}{3},1]}\frac{\frac{9p'^2-9p'+2}{3(2p'-1)}+\frac{1}{9p}}{2p'-p}
\right\}\\
=&\max\left\{ \frac{1}{9p^2},\frac{6p+1}{9p(2-p)}\right\},
\end{array}
\]
where the suprema are respectively attained
at $p$ and 1.
Solving
$\frac{1}{9p^2}=\frac{6p+1}{9p(2-p)}$,
we obtain
\[ a(p,\varphi_0)=\left\{
\begin{array}{ll}
\frac{1}{9p^2}&\mbox{ with }\rho(p,\varphi_0)=p, \mbox{ for }p<\bar p:=\frac{-1+\sqrt{13}}{6},\\
\frac{6p+1}{9p(2-p)}&\mbox{ with }\rho(p,\varphi_0)=1, \mbox{ for }p\geq\bar p.
\end{array}\right.
\]
Therefore  (using Step I.5) $p_1=\inf\{ p>p_0, \rho(p,\varphi_0)>p\}=\bar p$ and $w(p)=\varphi_0(p)=-\frac{1}{9p^2}$ for $p\in[\frac{1}{3},\bar p]$.
Finally (Step I.4) $p_2=\rho(\bar p,w)=1$ and $w(p)=-\frac{1}{9p}+\frac{1}{9p^2}(p-\bar p)$ on $[\bar p,1]$,
and the algorithm terminates.
In conclusion,
the limit value function is given by

\begin{equation}
\label{vexample1}
 v(p)=\left\{
\begin{array}{ll}
p-\frac{2}{3}, & \mbox{ if } 0\leq p<\frac{1}{3},\\
-\frac{1}{9p},  & \mbox{ if } \frac{1}{3}\leq p<\bar p, \mbox{  with } \bar p=\frac{\sqrt{13}-1}{6} (\simeq 0,434),\\
-\frac{1}{9\bar p}+\frac{1}{9\bar p^2}(p-\bar p), & \mbox{ if } \bar p\leq p\leq 1.
\end{array}
\right.
\end{equation}
In particular
\[ v(1)=\frac{1}{9\bar p^2}(2\bar p-1).\]

We will now solve two variants of Example~\ref{example3},
in which state $s_2$ is not absorbing and $\mu$ is not 1.

\noindent\textbf{Example~\ref{example3}, variation a.}

We here analyze the algorithm when $\lambda_1=3$ and $\lambda_2=r=1$, so that $\mu=\tfrac{1}{4}$ and $p^*=\frac{1}{4}$.

\begin{enumerate}
\item[1.] Initialization:
Simple calculations yield that
$\widetilde{p}_0=\frac{1}{3}$ and $p_0=1$.
By Eq.~\eqref{split} we obtain $w(p)=\frac{2}{15}(3p-2)$ for $p\in [0, \frac 13]$.
In particular $w(\frac 13)=-\frac 2{15}$.

\item[2.]
Now we have to compute the solution of Eq.~\eqref{varphiA} with the initial condition $\varphi(\frac 13)=-\frac 2{15}$.  For
$p\in [\frac 13,\frac 23]$,
the solution is
\[ \varphi(p)= -\frac {2}{15}3^{-\frac 14}(4p-1)^{-\frac 14},\]
and we obtain
\[ a(p,\varphi)=\max\left\{\frac{u(p)-\varphi(p)}{4p-1},\frac{u(1)-\varphi(p)}{4-p}\right\}=\max\left\{-\frac{\varphi(p)}{4p-1},\frac{\frac 23-\varphi(p)}{4-p}\right\}.\]
Following Step~I.5 of the algorithm, $p_1$ is the last  $p\in (\frac 13,1$] that satisfies the relation
\[ a(p,\varphi)=\frac{u(p)-\varphi(p)}{4p-1}.\]
On $[\frac 13,\frac 23]$, this relation is equivalent to
\[ (4p-1)^{-\frac 54}(1-p)=3^{\frac 14},\]
which yields $p_1\simeq 0.3858$.

\item[3.] The next step is to determine $p_2$ and the function $w$ on the interval $(p_1,p_2]$. As already noted, we have
\[ a(p_1,w)=\sup_{p\in[p_1,1]}\frac{u(p')-w(p_1)}{4p-p_1-1}=\frac{\frac 23-w(p_1)}{4-p_1},\]
and the supremum is
attained at $p'=1$.
Let $\Psi(p)=w(p_1)+(p-p_1)a(p_1,w)$, for every $p\in [p_1,1]$.
It can be shown that, for every $p\in (p_1,1)$ we have $a(p_1,w)>\frac{ u(p)-\Psi (p)}{4p-1}$. Therefore $p_2=1$, and, for every $p\in [p_1,1]$,
$w(p)=\Psi(p)$.
The first part of the algorithm ends.

\item[4.] Since $\widetilde p_0=0$, the second part of the algorithm
is vacuous.

\end{enumerate}

In conclusion,
the limit value function is given by

\begin{equation}
\label{vexample2}
v(p)=\left\{
\begin{array}{ll}
\frac{2}{15}(3p-2), & \mbox{ if } 0\leq p<\frac{1}{3},\\
-\frac 2{15}3^{-\frac 14}(4p-1)^{-\frac 14},  & \mbox{ if } \frac{1}{3}\leq p<p_1,\\
ap+b & \mbox{ if } p\in (p_1,1],
\end{array}
\right.
\end{equation}
where $a=\frac{\frac 23-w(p_1)}{4-p_1}\simeq 0.21709$ and $b=\frac 23-4a\simeq -0.20177$.

\bigskip

\noindent\textbf{Example~\ref{example3}, variation b.}

We here analyze the algorithm when $\lambda_1 = \tfrac{4}{3}$, $\lambda_2 = \tfrac{2}{3}$, and $r=1$, so that
$p^*=\frac{1}{3}$ and $\mu=\frac{1}{2}$.

\begin{enumerate}
\item[1.] Initialization:
Simple calculations show that
 $(\cav u)(\frac{1}{3})=u(\frac{1}{3})=0$
 and
$\widetilde{p}_0=p_0=\frac{1}{3}$.

\item[2.]
We turn to the first part of the algorithm.
The reader can verify that
\[ a(\tfrac{1}{3},w) =\sup_{p' \in (\frac{1}{3},1]}\frac{\mu(u(p')-w(\frac{1}{3}))}{ p'-\frac{1}{3}+\mu(p'-\frac{1}{3})}
=\max_{p' \in [\frac{2}{3},1]}\frac{\frac{9p'^2-9p'+2}{6p'-3}-0}{ 2(p'-\frac{1}{3})+(p'-\frac{1}{3})}
=\max_{p' \in [\frac{2}{3},1]}\frac{3p'-2}{6p'-3}.\]
This maximum is obtained at $p'=1$.
Therefore,
 $\rho(p_0,w)=1$, and the condition of Step~I.4 holds. We obtain $a(\frac{1}{3},w)=\frac{1}{3}$, and the first part of the algorithm ends with $p_1=1$.

\item[3.] For the second part of the algorithm we compute
\[ \widetilde{a}(\widetilde{p}_0,w)=\widetilde{a}(\tfrac{1}{3},w)=\inf_{p'\in[0,\frac{1}{3})}\frac{\mu(u(p)-w(\frac{1}{3}))}{p'-\frac{1}{3}+\mu(p'-\frac{1}{3})}
=\inf_{p'\in[0,\frac{1}{3})}\frac{\frac{9p'^2-9p'+2}{6p'-3}}{3p'-1}=\inf_{p'\in[0,\frac{1}{3})}\frac{3p'-2}{6p'-3}. \]

This infimum is obtained at $p'=0$.
Therefore,
 $\rho(\widetilde{p}_0,w)=0$, and the condition of Step~D.4 holds. We obtain $\widetilde{a}(\frac{1}{3},w)=\frac{2}{3}$, and the algorithm ends with $\widetilde{p}_1=0$.
\end{enumerate}

In conclusion, the limit value function is given by:

\begin{equation}
\label{vexample3}
 v(p)=\left\{
\begin{array}{ll}
\frac{1}{3}p-\frac{1}{9}, & \mbox{ if } 0\leq p<\frac{1}{3},\\
\frac{2}{3}p-\frac{2}{9},  & \mbox{ if } \frac{1}{3}\leq p<\bar 1,
\end{array}
\right.
\end{equation}
and it is not differentiable at $p^*=\frac{1}{3}$.

\subsection{The optimal strategy of the informed player}

The algorithm provided above allows one to describe the process $(p_t)_{t \geq 0}$ that attains the maximum in Eq.~\eqref{P1}.
As shown in \cite{CRRV},
this process
allows one to approximate an $\ep$-optimal strategy for the informed player in the game $G^{(n)}(p)$, provided $n$ is sufficiently large.
The $\ep$-optimal strategy depends on a parameter $q$, that changes along the play.
The initial value of $q$ is $p$, the initial belief of Player~2.
In stage $n$, if $q$ is an interior point of a
linear interval $[p_k,p_{k+1}]$,
then Player 1 reveals some information by performing a randomization which depends on his excess information.
This randomization changes Player 2's belief to either $p_k$ or $p_{k+1}$, with
appropriate probabilities. Using the terminology of \cite{CRRV}, over these
intervals, the strategy of Player~1 is revealing.
If, on the other hand, $q$ is an interior point of a nonlinear interval $[p_k,p_{k+1}]$,
then Player 1 reveals no information
(and plays the optimal strategy in the one-shot game among those that reveal no information).

\section{Proof of Theorem~\ref{wv}}\label{sec proof}

This section is devoted to the proof of Theorem~\ref{wv}.
In Section~\ref{sec:51} we study the sequence $(p_k)$ and show that the algorithm provided in Section~\ref{section:algorithm} terminates.
In Section~\ref{sec:52} we show that the function $w$ is concave and differentiable everywhere, except, possibly, at $p^*$.
In Section~\ref{sec:53} we show that $w=v$.

\subsection{On the sequence $(p_k)$}
\label{sec:51}

In this section we study the sequence $(p_k)$.
We will show that it is strictly increasing (Lemma~\ref{lemmapk})
and that if $p_{k}$ is defined by Eq.~\eqref{equ:i4} then $p_{k+1}$ is defined by Eq.~\eqref{equ:i5}, and vice versa
(Lemma~\ref{lemma:alternate}).
We will then show that there is $k \in \N$ such that $p_k=1$.
We start
with a technical lemma that will determine the value
  of $u$ and $\varphi$ on the elements of the sequence $(p_k)$.

\begin{Lemma}
    \label{extrema}
    Let $q\in(p^*,1)$. Suppose that
   $w$ is defined at $q$ and set $\rho:=\rho(q,w)$. Suppose that $u$ is twice differentiable on some open interval
    $I$ that contains $\rho$.
    \begin{enumerate}
        \item
     If $q<\rho$, then
            \begin{itemize}
                \item[i)] $u'(\rho)=\frac{1+\mu}{\mu}a(q,w)$,
                \item[ii)] $u''(\rho)\leq 0$.
            \end{itemize}
        \item Let
        $\varphi : [q,1] \to \R$ be a function satisfying
         $\varphi(q)=w(q)$ and $\varphi'(p)(p-p^*)=\mu\bigl(u(p)-\varphi(p)\bigr)$ on $[q,1]$. If $q=\rho$, then $\varphi''(q)\leq 0$.
        \end{enumerate}
\end{Lemma}

\begin{proof}
We start with the first claim.
 Set  $\Delta(p):=p-p^*+\mu(p-q)$ and
 $F(p):=\mu\cdot\frac {u(p)- w(q)}{\Delta(p)}$
for every
  $p\in I$.
 Since $u$ is differentiable
  on $I$, the function $F$ is also differentiable on $I$, and
    its derivative is
    \begin{equation}
    \label{fprim}
    F'(p)= \mu \cdot \frac{u'(p)\Delta(p)-(1+\mu)(u(p)-w(q))}{\Delta^2(p)}.
    \end{equation}
 If $q<\rho$,
 then
 $\rho$ is a local extremum in $I$ of $F$, and we have
    $F'(\rho)=0$.
From Eq.~\eqref{fprim} we obtain
    \begin{equation}
            \label{F00}
    u'(\rho)=(1+\mu)\frac {u(\rho)-w(q)}{\Delta(\rho)}= (1+\mu)\frac {u(\rho)-w(q)}{\rho-p^*+\mu(\rho-q)}.
    \end{equation}
Item {\em (i)} follows by the definition of $a(q,w)$.\\
The second derivative of $F$ at $\rho$ is
\begin{equation}
\label{equ:F''}
    F''(\rho)= \mu \cdot\frac{u''(\rho)\Delta^2(\rho)-2(1+\mu)\left(u'(\rho)\Delta(\rho)-(1+\mu)(u(\rho)-w(q))\right)}{\Delta^3(\rho)}.
\end{equation}
From Eq.~\eqref{F00} the second term in the numerator in Eq.~\eqref{equ:F''} vanishes,
hence
\[ u''(\rho)= \tfrac{1}{\mu}F''(\rho)\Delta(q). \]
Since $\rho$ is a local maximum of $F$, we have $F''(\rho)\leq 0$.
Since $\Delta(\rho) > 0$, item {\em (ii)} follows.

We turn to the second claim.
     If $\rho=q$, the maximum of $F$ on $[q,1]$ is attained
at $q$. Therefore $F'(q)\leq 0$. It follows from Eq.~\eqref{fprim} that
    \begin{equation}
    \label{u'}
     u'(q)\leq (1+\mu)\frac{u(q)-\varphi(q)}{q-p^*}=\frac{1+\mu}{\mu}\varphi'(q).
     \end{equation}
    Further, from the relation $\varphi'(q)(q-p^*)=\mu(u(q)-\varphi(q))$, we get
    \begin{equation}
    \label{equ:888}
    \varphi''(q)(q-p^*)=\mu u'(q)-(1+\mu)\varphi'(q).
    \end{equation}
    Eqs.~\eqref{u'} and~\eqref{equ:888} imply that
     $\varphi''(q)\leq 0$.
\end{proof}

\begin{Lemma}
    \label{lemmapk}
    For all
    $k \geq 0$ such that $p^*\leq p_k<1$, we have
    $p_k<p_{k+1}$.
\end{Lemma}

\begin{proof}
 Since the function $u$ is semi-algebraic, there exists
  $\epsilon>0$  such that $u$ is smooth on
  the interval
  $(p_k,p_k+\epsilon)$.

If $\rho(p_k,w)>p_k$, then
$(p_k,p_{k+1}]$ is a linear interval and $p_{k+1}=\rho(p_k,w)>p_k$: the claim is trivially satisfied.
If  $\rho(p_k,w)=p_k$, then
$(p_k,p_{k+1}]$ is a nonlinear interval
and
    \begin{equation}
    \label{pk+1=pk}
    p_{k+1}=\inf\{ p>p_k, \rho(p,\varphi)>p\},
    \end{equation}
 where $\varphi$ is the solution of Eq.~\eqref{varphiA}.
    In this case the result is not trivial. We shall prove it by contradiction.

Suppose to the contrary
that $p_{k+1}=p_k$. Then Eq.~\eqref{pk+1=pk} implies the existence of a sequence $(q^n)_{n\in\N}\subset (p_k,p_k+\epsilon)$ such that $q^n\searrow p_k$
and $\rho(q^n,\varphi)>q^n$ for every $n\in\N$.
In what follows, we set $\rho^n:=\rho(q^n,\varphi)$.

    Let $\bar\rho$ be an accumulation point of the sequence
    $(\rho^n)_{n \in \N}$
     and denote still by $(q^n)_{n\in\N}$ a subsequence of $(q^n)_{n\in\N}$
      that converges to $\bar\rho$.
    Since $p\mapsto a(p,\varphi)$ is continuous,%
    \footnote{Since we consider a nonlinear interval, $w=\varphi$ on $(p_k,p_{k+1}]$,
    hence $a(p,\varphi) = a(p,w)$ on that interval.}
    and since
    \begin{equation}
    \label{equ:889}
     a(q^n,\varphi)=\mu\cdot\frac{u(\rho^n)-\varphi(q^n)}{\rho^n-p^*+\mu(\rho^n-q^n)},
     \end{equation}
 letting $n$ tend to $\infty$ in Eq.~\eqref{equ:889}
     we get
    \[ a(p_k,\varphi)=\mu\cdot\frac{u(\bar\rho)-\varphi(p_k)}{\bar\rho-p^*+\mu(\bar\rho-p_k)}.\]
    By
     assumption, the value $a(p_k,\varphi)$ is attained only at $p_k=\rho(p_k,\varphi)$. Thus $\bar\rho=p_k$.
 By taking a subsequence of $(q^n)_{n \in \N}$, still denoted
  $(q^n)_{n \in \N}$, we can assume that $\rho^{n+1}<q^n<\rho^n$ for every $n\in\N$.

    By Lemma
     \ref{extrema}{(1)} we have
   $u''(\rho^n)\leq 0$ for every $n\in\N$.
    The function $u$
     is
     semi-algebraic,
      hence
    $u''(p)\leq 0$ on an interval $(p_k,q^{n_0})$ for some $n_0$ large enough. This implies that  $u'$ is nonincreasing  and $u$ is concave on $(p_k,q^{n_0})$.
    We can
     strengthen this conclusion
    : we can choose $n_0$ such that $u'$ is strictly decreasing and $u$ is strictly concave on  $(p_k,q^{n_0})$.
    Indeed, suppose that this
     does not hold. In this case $u$ is linear in a small one-sided neighborhood of $p_k$: there exist
    $\widetilde\epsilon\leq\epsilon$ and $\alpha,\beta\in\R$ such that $u(p)=\alpha p+\beta$ for all  $p\in [p_k,p_k+\widetilde\epsilon]$.
    By Lemma~\ref{extrema}{(1)},
    it follows that
    $\frac{\mu}{1+\mu}\alpha=a(q^n,\varphi)=a(p_k,\varphi)$ for $n$ sufficiently large.
    By Eq.~\eqref{equ:889}
     we therefore have
    \[ a(p_k,\varphi)=\mu\cdot\frac{\frac{1+\mu}{\mu}a(p_k,\varphi)\rho^n+\beta-\varphi(q^n)}{\rho^n-p^*+\mu(\rho^n-q^n)}\]
    or, equivalently, $\beta+\frac{a(p_k,\varphi)}{\mu}(p^*+\mu q^n)=\varphi(q^n)$.
     In addition, for every $p>\rho^n$ we have by the definition of $\rho^n$
    \[ a(p_k,\varphi)=a(q^n,\varphi)> \mu\cdot\frac{u(p) - \varphi(q^n)}{p-p^*+\mu(p-q^n)}=
    \mu\cdot\frac{\frac{1+\mu}{\mu}a(p_k,\varphi)p+\beta-\varphi(q^n)}{p-p^*+\mu(p-q^n)},\]
     or, equivalently,
    $\beta+\frac{a(p_k,\varphi)}{\mu}(p^*+\mu q^n)<\varphi(q^n)$,
    a contradiction. It follows that $u'$ is strictly decreasing in a small one-sided neighborhood of $p_k$.

Fix $n> n_0$.
    Since $u'$ is strictly decreasing on $(p_k,q^{n_0})$, we have
    $u'(\rho^{n})<u'(\rho^{n+1})$.
By Lemma~\ref{extrema}{(1)} applied to $q^n$ and $q^{n+1}$, we deduce that
$a(q^n,\varphi)<a(q^{n+1},\varphi)$.
    Since the function $p\mapsto a(p,\varphi)$ is continuous,
    there exists
     $q'\in  (q^{n+1},q^n)$ such that $\rho' := a(q',\varphi) \in (\rho^{n+1},\rho^n)$ and
    \begin{equation}
    \label{a'an}
    a(q^n,\varphi)<a(q',\varphi).
    \end{equation}
Consider now the function $\Psi$ on $[q',1]$ defined by $\Psi(p):=\varphi(q')+a(q',\varphi)(p-q')$.
 The reader can verify that the function
 $\Psi$ is a solution of
    \[
    \begin{array}{l}
    \Psi(q')=\varphi(q'),\\
    \Psi'(p)(p-p^*)=\mu(\bar u(p)-\Psi(p)), \; p\in[q',1],
    \end{array}
    \]
    with $\bar u(p)=\varphi(q')+\frac{a(q',\varphi)}{\mu}\bigl(p-p^*+\mu(p-q')\bigr)$
     for every
    $p\in [q',1]$.
    It follows that
     the function $\gamma : [q',1] \to \R$ defined by
 $\gamma(p):=\Psi(p)-\varphi(p)$ is a solution of
    \begin{equation}
    \label{gamma}
    \begin{array}{l}
    \gamma(q')=0,\\
    \gamma'(p)(p-p^*)=\mu(\widetilde u(p)-\gamma(p)), \; p\in(q',1),
    \end{array}
    \end{equation}
    with $\widetilde u(p)=\bar u(p)-u(p)$. Eq.~\eqref{gamma} can be solved quasi-explicitly:
    \[ \gamma(p)=c(p)(p-p^*)^{-\mu},\; p\in[q',1],\]
    with $c(q')=0$ and $c'(p)=\mu \widetilde u(p)(p-p^*)^{\mu-1}$.

    By the definition of $a(q',\varphi)$ and $\bar u$, the function $\widetilde u$ is nonnegative on $[q',1]$.
    It follows that, for every $p\in[q',1]$ we have $c(p)\geq 0$ and
 consequently
$\gamma(p)\geq 0$, which is equivalent to $\Psi(p) \geq \varphi(p)$.
Substituting
 $p=q^{n}$, we obtain in particular that $\Psi(q^{n})\geq \varphi(q^{n})$, or, equivalently,
    \begin{equation}
    \label{phipsi}
    \varphi(q^{n})-\varphi(q')\leq a(q',\varphi)(q^{n}-q').
    \end{equation}

To derive a contradiction, recall that, by the definition of $a(q',\varphi)$ and $a(q^{n},\varphi)$,
    \[ u(\rho')=\varphi(q')+\frac{a(q',\varphi)}{\mu}(\rho'-p^*+\mu(\rho'-q'))\]
    and
    \[ u(\rho') \leq\varphi(q^{n})+\frac{a(q^{n},\varphi)}{\mu}(\rho'-p^*+\mu(\rho'-q^{n})).\]
    Combining these two equations with Eq.~\eqref{phipsi}
we obtain
    \[
    a(q',\varphi)(\rho'-p^*+\mu(\rho'-q^{n})) \leq a(q^{n},\varphi)(\rho'-p^*+\mu(\rho'-q^{n})).\]
Since $\rho' > p^*$ and $\rho' > q^{n}$ this implies that
$a(q',\varphi) \leq a(q^{n},\varphi)$,
    contradicting Eq.~\eqref{a'an}.
It follows that $p_{k+1}$ that is defined by Eq.~\eqref{pk+1=pk} satisfies $p_{k+1} > p_k$.
\end{proof}

The following lemma says that linear intervals are followed by nonlinear intervals and vice versa.
\begin{Lemma}
\label{lemma:alternate}
    \begin{enumerate}
        \item If $p^*<p_0<1$, then $\rho(p_0,w)=p_0$ and
         $(p_0,p_1]$ is a nonlinear interval.
         (If $p_0=p^*$,
          then $(p_0,p_1]$ may be a linear or a nonlinear interval.)
        \item For every $k\geq 1$ such that $p_k>1$,
if $(p_{k-1},p_k]$ is a linear interval (resp.~a nonlinear interval), then $(p_k,p_{k+1}]$ is a nonlinear interval (resp.~a linear interval).
    \end{enumerate}
\end{Lemma}

\begin{proof}
     By the definition of $\widetilde p_0$ and $p_0$, we have
    \[ \frac{u(p)-u(p_0)}{p-p_0}<\frac{u(p_0)-u(\widetilde p_0)}{p_0-\widetilde p_0}, \ \ \ \forall p\in(p_0,1].\]
    Simple (though tedious) algebraic manipulations combining this inequality with Eq.~\eqref{split} for $p=p_0$ yield
    \[ \frac{u(p)-w(p_0)}{p-p^*+\mu(p-p_0)}<\frac{u(p_0)-w(p_0)}{p_0-p^*}, \ \ \ \forall p\in(p_0,1].\]
    Claim 1 follows.

We turn to prove Claim 2.
Suppose that
$(p_{k-1},p_k]$ is a linear interval.
By construction we have
\begin{equation}
\label{equ:21.1}
a(p_{k-1},w)=\mu\cdot\frac{u(p_{k})-w(p_{k-1})}{p_{k}-p^*+\mu(p_{k}-p_{k-1})}.
\end{equation}
Since on the interval $(p_{k-1},p_k]$ the function $w$ is defined by Eqs.~\eqref{equ:w:i4} and~\eqref{eq psi}, we have
$w(p_{k-1})=w(p_k)-a(p_{k-1},w)(p_k-p_{k-1})$, and Eq.~\eqref{equ:21.1} becomes
\begin{equation}\label{eq step i4}
a(p_{k-1},w)=\mu\cdot\frac{u(p_{k})-w(p_{k})}{p_{k}-p^*}.
\end{equation}
To show that
$(p_k,p_{k+1}]$ is a nonlinear interval we will show that $\rho_k:=\rho(p_{k},w)=p_k$.
By Eq.~\eqref{eq step i4} and Remark~\ref{remark alg}.\ref{remark:4} we have
\[  a(p_{k-1},w)=\mu\cdot\frac{u(p_{k})-w(p_{k})}{p_{k}-p^*}\leq a(p_k,w) = \mu\cdot\frac{u(\rho_k)-w(p_k)}{\rho_k-p^*+\mu(\rho_k-p_k)}.\]
Using again the relation $w(p_k)=w(p_{k-1})+a(p_{k-1},w)(p_k-p_{k-1})$, this last inequality becomes
\[ a(p_{k-1},w)\leq \mu\cdot\frac{u(\rho_k)-w(p_{k-1})}{\rho_k-p^*+\mu(\rho_k-p_{k-1})}.\]
Since $\rho(p_{k-1},w)$ is the maximal $p'$ that satisfies
$a(p_{k-1},w) = \mu\cdot\frac{u(p')-w(p_{k-1})}{p'-p^*+\mu(p'-p_{k-1})}$,
this implies that $\rho(p_{k},w)=\rho(p_{k-1},w)=p_k$, which is what we wanted to prove.
For later use we note that in this case we have  $a(p_{k-1},w)=a(p_k,w)$.

Finally assume that
$(p_{k-1},p_k]$ is a nonlinear interval,
so that $p_k=\inf\{ p>p_{k-1},\rho(p,w)>p\}$.
To prove that $(p_k,p_{k+1}]$ is a linear interval we will show that $\rho(p_k,w)>p_k$.
Suppose to the contrary that $\rho(p_k,w)=p_k$.
Then, the algorithm dictates that $p_{k+1}=\inf\{ p>p_{k},\rho(p,\varphi_k)>p\}$ and $w=\varphi_k$ on $(p_k,p_{k+1}]$.
By the definition of $p_k$, this implies that $p_{k+1}=p_k$, contradicting Lemma~\ref{lemmapk}.
We conclude that $(p_k,p_{k+1}]$ is a linear interval.
\end{proof}

\begin{Lemma}
    \label{lemma:e2}
    The algorithm ends after a finite number of iterations: there exists
 $k\geq 0$ such that $p_k=1$ and there exists $\widetilde{k}\geq 1$ such that $\widetilde{p}_{\widetilde{k}}=0$.
\end{Lemma}

\begin{proof}
We will prove the first claim.
The second claim is proven analogously.
Assume by contradiction that $p_k < 1$ for every $k \in \N$, and set $p_\infty=\lim_{n \to \infty}p_n$.
By Lemma \ref{lemmapk}, $(p_k,p_{k+1})\neq\emptyset$ for every $k\in\N$.
        Since $u$ is semi-algebraic, there is $n_0$ sufficiently large such that $u$ is twice differentiable on $[p_{n_0},p_\infty)$.
        Let $k\geq n_0$ be such that the interval $(p_k,p_{k+1}]$ is linear.
By Eq.~\eqref{eq stpA5} and the definition of $a(p_k,w)$ (Eq.~\eqref{equ:91}),
        \begin{itemize}
            \item $a(p_k,w)=\mu\cdot\frac{u(p_k)-w(p_k)}{p_k-p^*}
            =\mu\cdot\frac{u(p_{k+1})-w(p_{k})}{p_{k+1}-p^*+\mu(p_{k+1}-p_k)}$,
            \item $a(p_k,w)\geq \mu\cdot\frac{u(p)-w(p_k)}{p-p^*+\mu(p-p_k)}$, for every $p\in(p_k,p_{k+1})$.
        \end{itemize}
        Equivalently, if we set
        \begin{equation}
        \label{equ:39.1}
         f(p)=\mu(u(p)-w(p_k))-a(p_k,w)-a(p_k,w)(p-p^*+\mu(p-p_k)),
         \end{equation}
        it holds that $f(p_k)=f(p_{k+1})=0$ and $f(p)\leq 0$ for every $p\in(p_k,p_{k+1})$.

By Lemma~\ref{lemma:alternate} there are infinitely many linear intervals.
We argue now that, provided $k$ is sufficiently large,
if the interval $(p_k,p_{k+1}]$ is linear then
 there exists $p \in (p_k,p_{k+1})$ with $f(p) < 0$.
Indeed, if this is not true, then for every such $k$ sufficiently large, $f(p) = 0$ for every $p \in (p_k,p_{k+1})$.
By Eq.~\eqref{equ:39.1} this implies that $u$ is affine on $(p_k,p_{k+1})$.
Since $u$ is semi-algebraic, it is affine on the whole interval $[p_{n_1},p_\infty)$, for some large enough $n_1$.
But in this case, for every $k\geq n_1$, if $\rho(p_k,w)>p_k$ then $\rho(p_k,w)=p_\infty$,
contradicting the fact that $p_k < p_\infty$ for every $k$.

We conclude that for every $k$ sufficiently large such that
the interval $(p_k,p_{k+1}]$ is linear
 there is $p \in (p_k,p_{k+1})$ satisfying $f(p) < 0$.
In that case we can also find $p'\in(p_k,p_{k+1})$ such that $f''(p')>0$.
Since $f'' = \mu u''$, this implies that $u''(p_k) > 0$.
Since $u$ is semi-algebraic, this implies that $u''(p) > 0$ for every $p$ sufficiently close to $p_\infty$.
However, by Lemma~\ref{extrema} (1.ii), $u''(q)\leq 0$ for some $q\in(p_k,p_{k+1})$, a contradiction.
\end{proof}

\subsection{The function $w$ is differentiable and concave}
\label{sec:52}

\begin{Lemma}
\label{lemma:e}
The function $w$ is differentiable on $[0,p^*)\cup (p^*,1]$. If $\widetilde p_0<p^*<p_0$,
 then
$w$ is differentiable everywhere.
\end{Lemma}

\begin{proof}
By its definition, the function $w$ is linear on $[\widetilde p_0,p_0]$.
Hence $w$ is differentiable on $(p^*,p_0)$,
and if $\widetilde p<p^*<p_0$ then $w$ is
differentiable at $p^*$.

We next note that $w$ is differentiable on each interval $(p_{k-1},p_{k})$.
Indeed, on each of these intervals, $w$ is affine or the
solution of a standard first order differential equation.

We now show that $w$ is differentiable at each of the points $(p_k)_{k \geq 1}$,
and, if $p^* < p_0$, it is also differentiable at $p_0$.
Denote by $w'_-(p)$ (resp.~$w'_+(p)$) the left (resp.~right) derivative of $w$ at $p$.

If $p^*<p_0$, then  $w$ is affine on $[p^*, p_0]$ and
by Eq.~\eqref{split}
we have
$w'_-(p_0)=\frac{\mu(u(p_0)-u(\widetilde{p}_0))}{(p_0-\widetilde{p}_0)(\mu+1)}$.
By definition,
$w'_+(p_0)=\varphi'_+(p_0)=\mu\cdot\frac{u(p_0)-w(p_0)}{p_0-p^*}$.
Substituting $w(p_0)$
by its expression in Eq.~\eqref{split},
 we deduce that
$w'_+(p_0)=w'_-(p_0)$.

For $k \geq 1$, either $\rho(p_k,w)>p_k$ or $\rho(p_k,w)=p_k$.
In the first case, $(p_{k},p_{k+1}]$ is a linear interval,
and then $w'_{-}(p_k)=\mu\cdot\frac{u(p_k)-w(p_k)}{p_k-p^*}$
and $w'_{+}(p_k)=\mu\cdot\frac{u(p_{k+1})-w(p_k)}{p_{k+1}-p^*+\mu(p_{k+1}-p_k)}$.
$p_k$ is defined to be $\inf\left\{p>p_{k-1}: \rho(p,\varphi_{k-1})>p\right\}$.
Thus, for all $p\in(p_{k-1},p_k)$ we have

$$\frac{u(p_{k+1})-w(p)}{p_{k+1}-p^*+\mu(p_{k+1}-p)}\leq \frac{u(p)-w(p)}{p-p^*}.$$

Since $p_k$ is the infimum of a decreasing sequence where the last inequality is reversed, and by the continuity of $w$ and $u$, we get the equality of the derivatives.

In the second case
the interval $(p_{k-1},p_k]$ is linear.
For such $k$ we have  $w'_{-}(p_k)=a(p_{k-1},w)$ and $w'_{+}(p_k)=\mu\cdot\frac{u(p_k)-w(p_k)}{p_k-p^*}$,
and the equality of the derivatives follows from Eq.~\eqref{eq step i4}.

 Analogous
arguments hold for
 the interval
$[0,p^*)$.
\end{proof}

\begin{Lemma}
\label{lemma:concave}
 The function $w$ is concave.
\end{Lemma}

\begin{proof}

    On the interval $[\widetilde{p}_0,p_0]$ and on linear intervals the function $w$ is affine.
    We turn to prove that $w$ is concave on the nonlinear intervals.
 We will only discuss nonlinear intervals defined by Step~I.5.

         On nonlinear intervals
        the function $w$ coincides with the solution $\varphi$ of Eq.~\eqref{varphiA}.
        Moreover, $\rho(q,\varphi)=q$ for every $q$ in such an interval.
The function $u$ is semi-algebraic, hence twice differentiable on [0,1], except possibly at finitely many points.
If $q$ is in a nonlinear interval and $u$ is twice differentiable at $q$,
then $u$ is twice differentiable in an open neighborhood of $q$,
and hence, by Lemma \ref{extrema}(2), we have $w''(q)=\varphi''(q)\leq 0$.

It follows that the interval $[0,1]$ can be partitioned into finitely many subintervals such that
$w'$ is weakly decreasing on the interior of each of the subintervals.
If $w$ is differentiable on $[0,1]$, we can conclude from this that $w'$ is decreasing everywhere, i.e., $w$ is convex on the whole interval $[0,1]$.
By Lemma \ref{lemma:e}, this is the case when $\widetilde p_0<p^*<p_0$.
If $p^* = 0$ (resp.~$p^*=1$) then $w$ is concave on $(0,1]$ (resp.~$[0,1)$), and hence also on $[0,1]$.

Suppose then that $p^* \in (0,1)$, and $p^*=p_0$ or $p^*=\widetilde p_0$.
In this case we have to examine the behavior of $w$ at $p^*$, where it may not be differentiable.
We will handle the case $p^* = p_0$.
The case $p^* = \widetilde p_0$ is solved analogously.

Since $w$ is differentiable on $[0,p^*) \cup (p^*,1]$,
both the left and the right derivatives at $p^*$ exist, and it is sufficient to show that $w'_+(p^*)\leq w'_-(p^*)$.
We will show that $w'_+(p^*) = a(p^*,w)$.
Indeed, when $p^*=p_0$ we have  $w(p^*)=u(p^*)$, and therefore
        \[ a(p^*,w)=\frac{\mu}{1+\mu}\sup_{p\in(p^*,1]}\frac{u(p)-u(p^*)}{p-p^*}. \]
We distinguish between two cases.
\begin{itemize}
\item  If $\rho(p^*,w)>p^*$, then
the interval $[p^*,p_1]$ is linear,
 that is, $w=\psi_0$, and we have
        \[ w'_+(p^*)=\psi_{0+}'(p^*)=a(p^*,w).\]
\item If $\rho(p^*,w)=p^*$, then
the interval $[p^*,p_1]$ is nonlinear. In this case, we have
\begin{eqnarray}
\begin{array}{rl}
\nonumber
 w'_+(p^*)=\varphi_{0+}'(p^*)=&\lim_{p\searrow p^*}\varphi_0'(p)
 =\lim_{p\searrow p^*}\mu\cdot\frac{u(p)-\varphi_0(p)}{p-p^*}\\
 =&\mu\lim_{p\searrow p^*}\frac{u(p)-u(p^*)}{p-p^*}-\mu\lim_{p\searrow p^*}\frac{\varphi_0(p)-\varphi_0(p^*)}{p-p^*}\\
 \label{equ:8:1}
 =&\mu u'_+(p^*) - \mu \varphi_{0+}'(p^*)
 =\mu u'_+(p^*) - \mu w'_+(p^*).
 \nonumber
 \end{array}
 \end{eqnarray}
 It follows from Eq.~\eqref{equ:8:1} that
 \[ w'_+(p^*)=\frac{\mu}{1+\mu}u'_+(p^*)=\frac{\mu}{1+\mu}\sup_{p\in(p^*,1]}\frac{u(p)-u(p^*)}{p-p^*} = a(p^*,w).\]
 \end{itemize}
 We now calculate $w'_-(p^*)$.
If $\widetilde p_0=p_0=p^*$, a similar argument shows that $w'_-(p^*)=\widetilde a(p^*,w) = \frac{\mu}{1+\mu}\inf_{p\in[0,p^*)}\frac{u(p)-u(p^*)}{p-p^*} $.
However, in this case, at $p_0$ the function $u$ is equal to its convex hull, and therefore
\[ w'_+(p^*) = \frac{\mu}{1+\mu}\sup_{p\in(p^*,1]}\frac{u(p)-u(p^*)}{p-p^*}\leq
\frac{\mu}{1+\mu}\inf_{p'\in[0,p^*)}\frac{u(p')-u(p^*)}{p'-p^*} = w'_-(p^*),\]
as desired.
If $\widetilde p_0<p_0=p^*$, Eq.~\eqref{split} yields
\[ w'_-(p^*)=\frac{\mu}{1+\mu}\frac{u(p_0)-u(\widetilde p_0)}{p_0-\widetilde p_0}.\]
From the definition of $\widetilde p_0$ and $p_0$ we deduce that
        \begin{equation}
        \label{ineqp0}
\frac{u(p_0)-u(\widetilde p_0)}{p_0-\widetilde p_0}\geq \sup_{p\in(p_0,1]}\frac{u(p)-u( p_0)}{p-p_0},
        \end{equation}
and once again it follows that $w'_+(p^*)\leq w'_-(p^*)$.
\end{proof}

\subsection{The functions $w$ and $v$ coincide. }
\label{sec:53}

\begin{Proposition}
\label{lemma:f}
 For every $p\in[0,1]$ we have $w(p)=v(p)$.
\end{Proposition}

\begin{proof}
To prove the claim we show that the function $w$ satisfies the conditions of Theorem~\ref{Char}.
Condition G.1 holds by the definition of $w$ on the interval $[\widetilde p_0,p_0]$.
By Lemmas~\ref{lemma:e} and~\ref{lemma:concave},
$w$ is concave and differentiable on $[0,1]\setminus\{ p^*\}$.

Since $w$ is affine
 on the interval $[\widetilde{p}_0,p_0]$ and on  linear intervals $[p_k,p_{k+1}]$,
and since by Lemma~\ref{lemma:alternate} the two end-points of these intervals lie in
nonlinear intervals,
it follows that
all the extreme points of the hypograph of $w$ lie in nonlinear intervals $[p_k,p_{k+1}]$.
On these intervals, from Eq.~\eqref{eq stpA5}, the relation $w'(p)(p-p^*) + \mu\left(w(p)-u(p)\right)=0$ holds,
 and therefore Condition G.3 holds.
Moreover, Condition G.2 holds on  nonlinear intervals $[p_k,p_{k+1}]$.

It remains to show that
Condition G.2 holds:
$w'(p)(p-p^*) + \mu\left(w(p)-u(p)\right)\geq 0$
 on the interval $[\widetilde{p}_0,p_0]$ and on linear intervals.
On a linear interval $(p_k,p_{k+1}]$ we have $w(p)=w(p_{k})+(p-p_{k})a(p_{k},w)$ and $w'(p)=a(p_{k},w)$.
It then follows by the definition of $a(p_k,w)$ that
 on these intervals
\[\begin{array}{rl}
(p-p^*)w'(p)+\mu(w(p)-u(p))
=&(p-p^*)a(p_{k},w)+\mu\bigl(w(p_{k})+(p-p_{k})a(p_{k},w)-u(p)\bigr)\\
=&\bigl(p-p^*+\mu(p-p_k)\bigr)a(p_{k},w)+\mu \bigl(w(p_{k})-u(p)\bigr)\\
 \geq &0,
    \end{array}\]
as desired.

On the interval
$[\widetilde{p}_0,p_0]$ the function $w$ is affine, thus $w'$ is constant and
therefore the function
$w(p)+ \frac{w'(p)(p-p^*)}{\mu}$ is affine as well.
The points $(\widetilde{p}_0,u(\widetilde{p}_0))$ and $(p_0,u(p_0))$ are on
the graph of this last function.
These points and the
interval
 connecting them are on
the graph of the function
$\cav u$.
It follows that for every
$p\in (\widetilde{p}_0, p_0)$
we have
 $u(p)\leq w(p)+ \frac{w'(p)(p-p^*)}{\mu}$, which implies
 that
 $w'(p)(p-p^*) + \mu\left(w(p)-u(p)\right)\geq 0$.

\end{proof}


\begin{thebibliography}{abc99xyz}

\bibitem{AH}
Aumann R. J., Heifetz A.
\it{Incomplete information},
In Handbook of Game Theory with Economic Applications, Volume 3, (2002), Chapter 43, 1665--1686.

\bibitem {aum}
{ Aumann, R. J., Maschler, M. B.}, {\it Repeated games of
incomplete information: The zero-sum extensive case}, in Report of the
U.S. Arms Control and Disarmament Agency ST-143, Washington, D.C., 1968,
Chapter III, pp. 37--116.

\bibitem{AM} Aumann, R. J., Maschler, M. B. {\it Repeated games with incomplete information}, MIT press, (1995).

\bibitem {aum3}
{ Aumann, R. J., Maschler, M. B., Stearns, R. E.},
{\it Repeated games of incomplete information: An approach to the non-zero sum
case}, in Report of the U.S. Arms Control and Disarmament Agency ST-143,
Washington, D.C., 1968, Chapter IV, 117--216.

\bibitem{c1} Cardaliaguet P.
{\it Differential games with asymmetric information.} SIAM J. Control Optim. (2006) 46, no. 3, 816-838.

\bibitem{c3} Cardaliaguet P. {\it Numerical approximation and optimal strategies for differential games with lack of information on one side.} Advances in dynamic games and their applications, 159–-176,
Ann. Internat. Soc. Dynam. Games, 10, Birkhäuser Boston, Inc., Boston, MA, 2009.

\bibitem{cr1} Cardaliaguet P., Rainer C. {\it Stochastic differential games with asymmetric information}, Appl. Math. Optim. 59 (2009), no. 1,1--36.

\bibitem{cr2} Cardaliaguet P., Rainer C., {\it  On a continuous-time game with incomplete information},
Mathematics of Operations Research, (2009), 34(4), 769--794.


\bibitem{CRRV} Cardaliaguet P. Rainer C., Rosenberg D., Vieille N., {\it  Markov games with frequent actions and incomplete information},
Mathematics of Operations Research, (2016), 41(1), 49--71.

\bibitem{Gensbittel1} Gensbittel, F., {\it Continuous-time limit of dynamic games with incomplete information and a more informed player},
International Journal of Game Theory,  (2016), 45(1-2), 321--352.

\bibitem{Gensbittel} Gensbittel F., {\it Continuous-time Markov games with asymmetric information}, arXiv:1802.08536.

\bibitem{gruen} Gr\"{u}n C. {\it A BSDE approach to stochastic differential games with incomplete information.} Stochastic Process. Appl. (2012), 122(4), 1917-–1946.

\bibitem{gruen3} Gr\"{u}n C. {\it A probabilistic-numerical approximation for an obstacle problem arising in game theory.} Appl. Math. Optim. 66 (2012), no. 3, 363–-385.

\bibitem{Harsanyi} Harsanyi J. C., {\it Games with incomplete information played by ``Bayesian'' players, I--III Part I. The basic model},
Management science, (1967), 14(3), 159--182.

\bibitem{HRSV}
H\"orner J., Rosenberg D., Solan E., Vieille N.,
On a Markov game with one-sided information.
\emph{Operations Research}, (2010), 58, 1107--1115.

\bibitem{MSZ}
Mertens J.-F., Sorin S., Zamir S.,
\it{Repeated Games}, Cambridge University Press, (2016).

\bibitem{MZ} Mertens J. F., Zamir S., {\it The value of two-person zero-sum repeated games with lack of information on both sides},
International Journal of Game Theory, (1971), 1(1),39--64.


\bibitem{Ney}
Neyman A.,
Existence of optimal strategies in Markov games with incomplete information.
\emph{International Journal of Game Theory}, (2008), 37(4), 581--596.

\bibitem{oliu} Oliu-Barton, M, {\it Differential games with asymmetric and correlated information}, Dyn. Games Appl.  (2015), 5(3), 378–-396.

\bibitem{Renault}
Renault J.,
The value of Markov chain games with lack of information on one side, \emph{Mathematics of Operations Reserach}, (2006),
{31}, 490--512.

\end{thebibliography}
\end{document}